\documentclass[12pt]{amsart}

\usepackage{amssymb,amsmath,graphicx,color,textcomp, amsthm,bbm,bbold, enumerate,booktabs}
\usepackage{chngcntr}
\usepackage{apptools}
\usepackage{mathrsfs}
\usepackage{tikz}
\usetikzlibrary{trees}

\AtAppendix{\counterwithin{theorem}{section}}
\definecolor{Red}{cmyk}{0,1,1,0}

\definecolor{verde}{cmyk}{1,0,1,0}

\definecolor{loka}{cmyk}{.5,0,1,.5}
\definecolor{azul}{cmyk}{1,1,0,0}

\numberwithin{equation}{section}

\newcommand{\be}{\begin{equation}}
\newcommand{\ee}{\end{equation}}

\newtheorem{theorem}{Theorem}

\newtheorem{definition}{Definition}
\newtheorem{lemma}{Lemma}
\newtheorem{remark}{Remark}

\begin{document}
\title{Leibniz type rule: $\Psi-$Hilfer fractional derivative}
\author{J. Vanterler da C. Sousa$^1$}
\address{$^1$ Department of Applied Mathematics, Institute of Mathematics,
 Statistics and Scientific Computation, University of Campinas --
UNICAMP, rua S\'ergio Buarque de Holanda 651,
13083--859, Campinas SP, Brazil\newline
\newline
e-mail: {\itshape \texttt{ra160908@ime.unicamp.br, capelas@ime.unicamp.br}}}

\author{E. Capelas de Oliveira$^1$}

\begin{abstract} In this paper, we present the Leibniz rule for the $\Psi-$Hilfer ($\Psi-$H) fractional derivative in two versions, the first in relation to $\Psi-$RL fractional derivative and the second in relation to the $\Psi-$H fractional derivative. In this sense, we present some particular cases of Leibniz rules and Leibniz type rules from the investigated case.

\end{abstract}
\maketitle

\section{Introduction}
The fractional calculus (FC) is the calculation of modernity, where the limitation of the integration and derivation orders to the integers is overcome and the orders of these operators are arbitrarily assigned. FC has become the target of studies around the world, and is not currently considered a highly researched subject, by chance, as the number of articles published in the area and its relationship with other areas has been growing every decade \cite{SAMKO,KSTJ,POD,machado,mainardi,debnath,butzer,por,por1}. 

The FC appeared in the year 1695 \cite{GWL1,GWL2,GWL3}, in an exchange of correspondence between l'H\^opital and Leibniz, even during the construction of the 
classic differential and integral calculus. In these correspondences Leibniz urged the possible generalization of the whole-order derivative to an arbitrary order, 
l'H\^opital then questioned him about the special case where the order of the derivative was 1/2. In the reply letter, dated September 30, 1695, Leibniz presented 
a correct reflection, in which he affirmed that very important consequences would come from these developments \cite{GWL1,GWL2,GWL3}.

Encouraged by this new perspective of FC application, several authors have developed definitions for fractional derivatives and integrals in subsequent decades, 
but some of these definitions have contradicted each other. One of these definitions, which emerged in the nineteenth century; is the proposal by Liouville, 
which was later reformulated by Riemann, promoting the important definitions of the Riemann-Liouville derivative and integral \cite{SAMKO,KSTJ,POD}. 
In this sense, Caputo introduced the so-called Caputo fractional derivative, fudamental in the study of memory effects and also in the modeling of real problems, 
by means of differential equations, since the initial conditions, when using Laplace transform to obtain the solution to analysis of this differential 
equstion, are integers \cite{POD,mainardi,jose,jose1,ten,Amer,Giusti}. 

Among other applications that over time were justifying the relevance of the fractional derivative, numerous definitions of fractional derivatives appeared, 
among which we mention: Hadamard, Weyl, Caputo-Hadamard, Katugampola, Caputo-Katugampola, Hilfer, Hilfer-Hadamard, Hilfer-Katugampola, Jumarie, Erd\'elyi-Kober, 
Riesz, Caputo–Riesz, Cassar, Gr\"unwald-Letnikov, each with its respective importance and application \cite{SAMKO,KSTJ,POD,review,VanterlerOliveira2,jarad,almeida,oliveira}. 
However, when we consider non-integer order derivatives, we have been able, in some studies, to better adapt the theoretical model to the experimental data, 
thus predicting better the future dynamics of the process. One problem that arises in this study is the innumerable definitions of fractional operators, and 
with this it must be taken into account the choice of the best operator for the case under study. One way to overcome this problem is to consider more general 
definitions, where the usual derivatives are particular cases. In this sense, Sousa and Oliveira, introduced the so-called $\Psi-$H fractional derivative 
which contains a wide class of fractional derivatives as particular cases \cite{VanterlerOliveira2}.

Given the large number of definitions of fractional derivatives, the natural question that arises is, what are the conditions that a given differentiation operator 
must satisfy to be considered fractional? Then in 2015 Ortigueira and Machado \cite{ortigueira} published an article, presenting a criterion for a certain differentiable 
operator, to be considered fractional. However, some of those derivatives that have been introduced so far do not satisfy all the conditions of the criterion. The main 
one that many fractional derivatives do not satisfy is the Leibniz rule.

Osler presented some works related to the rule of Leibniz \cite{osler,osler1,osler2,osler3}. On the other hand, Tarasov has also devoted himself over time to investigate 
and present interesting works on the Leibniz rule as well as worked on the violation of the Leibniz rule \cite{tarasov,tarasov1,tarasov2,tarasov3}. 
In this sense, in 2018 Sayevand, Machado and Baleanu, presented a new look on the Leibniz rule for the fractional derivative \cite{sayevand}.

As we already said, Sousa and Oliveira \cite{VanterlerOliveira2} proposed a new fractional derivative called $\Psi-$H. In this article we investigated four of the five 
conditions relative to the criterion to be a fractional derivative and other innumerable properties, lemmas, theorems, examples, as well as rules of fractional derivatives. 
However, the Leibniz rule was still lacking.

Thus, with this criterion to be determined, we have as main objective of this article, to investigate the Leibniz type rule, given by
\begin{eqnarray*}
^{\mathbf{H}}\mathfrak{D}_{a+}^{\alpha ,\beta ;\Psi }\left( fg\right) \left( x\right) &=&\overset{\infty }{\underset{m=0}{\sum }}\left( 
\begin{array}{c}
\alpha  \\ 
m
\end{array}%
\right) f^{\left( m\right) }\left( x\right) \text{ }^{\mathbf{H}}\mathfrak{D}_{a+}^{\alpha -m,\beta ;\Psi }g\left( x\right) + \Omega_{f,g}(\alpha,\beta,a),
\end{eqnarray*}
with 
$$
\Omega_{f,g}(\alpha,\beta,a)=\underset{k=0}{\overset{\infty }{\sum }}\left( \begin{array}{c} -\xi  \\ k\end{array} \right) 
\mathcal{I}_{a+}^{\xi +k;\Psi }g\left( a\right) \left( f^{\left( k\right) }\left( x\right) -f^{\left( k\right) }\left( a\right) 
\right) \dfrac{\left( \Psi \left( x\right) -\Psi \left( a\right) \right) ^{\xi-\alpha}}{\Gamma \left( \beta -\alpha \beta \right) },
$$ 
being $\xi = (1-\beta)(1-\alpha)$, $^{\mathbf{H}}\mathfrak{D}_{a+}^{\alpha ,\beta ;\Psi }\left( \cdot\right)$ is the $\Psi-$H fractional derivative. 
In addition, we show that from this Leibniz rule, we obtain their respective Leibniz rule and Leibniz type rule, since the $\Psi-$H 
fractional derivatives contemplates a vast class of fractional derivatives.

The article has the following structure: In section 2, we present the definitions of the Riemann-Liouville fractional integral with respect to another 
function and the $\Psi-$H fractional derivative. We analyze the lemmas \ref{lemma1} e \ref{lemma2}, which are fundamental in the investigation of the Leibniz type rule.
In section 3, we investigate the main result of the article, the Leibniz rule. Two versions for the Leibniz-type rule are presented, one version in relation to the 
$\Psi-$RL fractional derivative and the other in relation to the $\Psi-$H fractional derivative. In this sense, we present some Leibniz rules and Leibniz type rules 
from the choice of the $\Psi(\cdot)$ function and the limits $\beta \rightarrow 1$ or $\beta \rightarrow 0$. Concluding remarks close the paper.

\section{Preliminary results}

In this section, we present the definitions of the Riemann-Liouville fractional integral with respect to another function and the $\Psi-$H fractional derivative. 
We also investigate two Lemmas \ref{lemma1} and \ref{lemma2}, fundamental in obtaining the main result of the article.

\begin{definition}{\rm \cite{VanterlerOliveira2}} Let $I:=\left( a,b\right) $ with $-\infty \leq a<b\leq \infty $ an interval in a real line, $\alpha >0$ and $\Psi 
\left( \cdot \right)$ an increased and positive monotone function in $\left( a,b\right] $, whose derivative is continuous in $\left( a,b\right)$. The $\Psi-$RL fractional integral
to the left, of the order $\alpha$, of a function $f$, in relation to the function $\Psi$, is given by
\begin{equation}\label{eq1}
\mathcal{I}_{a+}^{\alpha ;\Psi }f\left( x\right) =\frac{1}{\Gamma \left( \alpha \right) }\int_{a}^{x}\mathfrak{N}(x,t)f\left( t\right) dt
\end{equation}
with $\mathfrak{N}(x,t):=\Psi ^{\prime }\left( t\right) \left( \Psi \left( x\right) -\Psi \left( t\right) \right) ^{\alpha -1}$ and the line $\Psi'(\cdot)$ denotes 
the derivative with respect to the variable $t$.
\end{definition}

In an analogous way, the $\Psi-$RL fractional integral on the right \cite{VanterlerOliveira2}.

Note that the integral Eq.(\ref{eq1}) can be written as follows, through equation
\begin{equation}
\mathcal{I}_{a+}^{\alpha ;\Psi }f\left( x\right) =Q_{\Psi }\mathcal{I}_{a+}^{\alpha ;\Psi
}Q_{\Psi }^{-1}f\left( x\right)
\end{equation}
with $Q_{\Psi }f\left( x\right) =f\left( \Psi \left( x\right) \right) $ and $Q_{\Psi }^{-1}$ denotes the inverse operator.

\begin{definition} {\rm\cite{VanterlerOliveira2}} Let $n-1<\alpha <n$ with $n\in \mathbb{N}$ and $\Lambda =\left[ a,b\right]$ an interval $-\infty 
\leq a<b\leq \infty $, $\Psi $ an increasing function such that $\Psi ^{\prime }\left( x\right) \neq 0$, $\forall t\in \Lambda $ and $\Psi \in C^{n}\left(\Lambda, 
\mathbb{R}\right) $ and $f\in C^{n}\left( \Lambda ,\mathbb{R}\right) $. The $\Psi-$H fractional derivative of order $\alpha $ and type $\beta $ with $0\leq \beta \leq 1$, 
to the left of a function $f$, is given by
\begin{equation}\label{eq3}
^{\mathbf{H}}\mathfrak{D}_{a+}^{\alpha ,\beta ;\Psi }f\left( x\right) =\mathcal{I}_{a+}^{\beta \left(n-\alpha \right) ;\Psi }\left( \frac{1}{\Psi ^{\prime }\left( x\right) }
\frac{d}{dx}\right) ^{n}\mathcal{I}_{a+}^{\left( 1-\beta \right) \left( n-\alpha\right) ;\Psi }f\left( x\right).
\end{equation}
\end{definition}
In an analogous way, the $\Psi-$H fractional derivative on the right \cite{VanterlerOliveira2}.

We can write the $\Psi-$H fractional derivative in terms of the $\Psi-$Riemann-Liouville ($\Psi-$RL) fractional derivative and $\Psi-$Caputo ($\Psi-$C) fractional derivative, as follow
\begin{equation}\label{eq4}
^{\mathbf{H}}\mathfrak{D}^{\alpha ,\beta ;\Psi }f\left( x\right) =\mathcal{I}_{a+}^{\beta \left(
n-\alpha \right) ;\Psi }\text{ }^{\mathbf{Rl}}\mathfrak{D}^{\alpha +\beta \left( n-\alpha
\right) ;\Psi }f\left( x\right)
\end{equation}
a
\begin{equation}\label{eq5}
^{\mathbf{H}}\mathfrak{D}^{\alpha ,\beta ;\Psi }f\left( x\right) =\text{ }^{\mathbf{C}}\mathfrak{D}^{\beta
\left( \alpha -n\right) +n;\Psi }\mathcal{I}_{a+}^{\left( 1-\beta \right) \left(
n-\alpha \right) ;\Psi }f\left( x\right),
\end{equation}
respectively.

The $\Psi-$H fractional derivatives generalizes a wide class of fractional derivatives of these, we mention: $\Psi-$RL, $\Psi-$C, 
Weyl, Chen, Jumarie, Hilfer, Hilfer-Katugampola. For other specific cases, we suggest \cite{VanterlerOliveira2}.

As we have already said, the Lemma \ref{lemma1} is important in obtaining the Leibniz type rule as we shall see below.

\begin{lemma}\label{lemma1} Let $\left( a,b\right] $ with $-\infty \leq a<b\leq \infty $ an interval in real line, $\alpha >0$ and $\Psi \left( x\right)$ 
a monotone growing and positive function in $\Lambda$, whose derivative is continuous in $\left( a,b\right)$. Thus,
\begin{equation}\label{eq6}
\mathcal{I}_{a+}^{\alpha ;\Psi }f\left( x\right) =\underset{n=0}{\overset{\infty }{%
\sum }}\left( 
\begin{array}{c}
-\alpha \\ 
n%
\end{array}%
\right) f^{\left( n\right) }\left( x\right) \frac{\left( \Psi \left( x\right) -\Psi \left( a\right) \right) ^{\alpha +n}}{\Gamma \left( \alpha +n+1\right) }
\end{equation}
where $f^{\left( n\right) }$ is the $n$-th derivative of the entire order and $x>a$.
\end{lemma}

\begin{proof}
In fact, suppose we can write the function $f$ as follows
\begin{equation} \label{eq7}
f\left( t\right) =\underset{n=0}{\overset{\infty }{\sum }}\frac{f^{\left( n\right) }\left( x\right) }{n!}\left( \Psi \left( t\right) -\Psi \left(x\right) \right) ^{n}.
\end{equation}

Substituting the function $f$, Eq.(\ref{eq7}), in the $\Psi-$RL fractional integral, we get
\begin{eqnarray}\label{eq8}
\mathcal{I}_{a+}^{\alpha ;\Psi }f\left( x\right) &=&\frac{1}{\Gamma \left( \alpha \right) }\int_{a}^{x}\Psi ^{\prime }\left( t\right) 
\left( \Psi \left( x\right) -\Psi \left( t\right) \right) ^{\alpha -1}\underset{n=0}{\overset{ \infty }{\sum }}\frac{f^{\left( n\right) }
\left( x\right) }{n!}\left( \Psi\left( t\right) -\Psi \left( x\right) \right) ^{n}dt  \notag \\
&=&\frac{1}{\Gamma \left( \alpha \right) }\overset{\infty }{\underset{n=0}{\sum }}\frac{f^{\left( n\right) }\left( x\right) 
\left( -1\right) ^{n}}{n!}\int_{a}^{x}\Psi ^{\prime }\left( t\right) \left( \Psi \left( x\right) -\Psi \left( t\right) \right) ^{\alpha -1+n}dt  \notag \\
&=&\frac{1}{\Gamma \left( \alpha \right) }\overset{\infty }{\underset{n=0}{\sum }}\frac{f^{\left( n\right) }\left( x\right) 
\left( -1\right) ^{n}}{n!}\frac{\left( \Psi \left( x\right) -\Psi \left( a\right) \right) ^{\alpha +n}}{\alpha +n}  \notag \\
&=&\frac{1}{\Gamma \left( \alpha \right) }\overset{\infty }{\underset{n=0}{\sum }}\frac{f^{\left( n\right) }\left( x\right) 
\left( -1\right) ^{n}}{n!}\frac{\Gamma \left( \alpha +n\right) \left( \Psi \left( x\right) -\Psi \left( a\right) \right) ^{\alpha +n}}{\Gamma \left( \alpha +n+1\right) }.
\end{eqnarray}

Note that from the definition of binomial coefficient and the reflection formula for the gamma function, we have
\begin{equation*}
\left( 
\begin{array}{c}
-\alpha \\ 
n
\end{array}
\right) =\frac{\Gamma \left( 1-\alpha \right) }{n!\Gamma \left( 1-\alpha -n\right) }=\frac{\pi }{\sin \left( \pi \alpha \right) \Gamma \left( \alpha \right) 
\left( -\alpha -n\right) \Gamma \left( -\alpha -n\right) n!}.
\end{equation*}

On the other hand, we can write
\begin{equation*}
\Gamma \left( -\alpha -n\right) \Gamma \left( 1+\alpha +n\right) =\frac{\pi }{\sin \left( -\pi \left( \alpha +n\right) \right) }=\frac{-\pi }{\sin 
\left( \pi \alpha \right) \left( -1\right) ^{n}}.
\end{equation*}

Then, we obtain
\begin{equation*}
\left( 
\begin{array}{c}
-\alpha \\ 
n
\end{array}
\right) =\frac{-\pi \Gamma \left( 1+\alpha +n\right) \left( -1\right) ^{n}\sin \left( \alpha \pi \right) }{\sin \left( \pi \alpha \right) 
\Gamma \left( \alpha \right) \left( -\alpha -n\right) n!\pi }=\frac{\left(
-1\right) ^{n}\Gamma \left( \alpha +n\right) }{\Gamma \left( \alpha \right)n!}.
\end{equation*}

Thus, by means of Eq.(\ref{eq8}), we conclude
\begin{equation*}
\mathcal{I}_{a+}^{\alpha ;\Psi }f\left( x\right) =\underset{n=0}{\overset{\infty }{%
\sum }}\left( 
\begin{array}{c}
-\alpha \\ 
n%
\end{array}%
\right) f^{\left( n\right) }\left( x\right) \frac{\left( \Psi \left(
x\right) -\Psi \left( a\right) \right) ^{\alpha +n}}{\Gamma \left( \alpha
+n+1\right) },
\end{equation*}
which is the desired result. \end{proof}

The presentation of Lemma \ref{lemma2} below is also of paramount importance for the investigation of the Leibniz type rule as will be seen in the next section.

\begin{lemma}\label{lemma2} Let $\left( a,b\right] $ with $-\infty \leq a<b\leq \infty $ an interval in real line, $\alpha >0$ and $\Psi \left( x\right) $ 
a monotone growing and positive function in $\Lambda$, whose derivative is continuous in $\left( a,b\right)$. The left fractional integral of a two-function product is given by
\begin{equation*}
\mathcal{I}_{a+}^{\alpha ;\Psi }\left( fg\right) \left( x\right) =\underset{k=0}{\overset{\infty }{\sum }}\left( 
\begin{array}{c}
-\alpha \\ 
k
\end{array}
\right) f^{\left( k\right) }\left( x\right) \mathcal{I}_{a+}^{\alpha +k;\Psi }g\left(x\right) .
\end{equation*}
where $f^{\left( k\right) }$ is the $k$-th derivative of the entire order and $x>a$.
\end{lemma}

\begin{proof}
Let $f$ and $g$ be two functions satisfying the conditions of Lemma \ref{lemma1}. Then, Eq.(\ref{eq6}) can be written in the following form
\begin{equation}\label{eq9}
\mathcal{I}_{a+}^{\alpha ;\Psi }\left( fg\right) \left( x\right) =\overset{\infty }{\underset{n=0}{\sum }}\left( 
\begin{array}{c}
-\alpha \\ 
n%
\end{array}%
\right) \left( fg\right) ^{\left( n\right) }\left( x\right) \frac{\left(\Psi \left( x\right) -\Psi \left( a\right) \right) ^{\alpha +n}}{\Gamma\left( \alpha +n+1\right) }.
\end{equation}

Using the Leibniz rule for the whole case in Eq.(\ref{eq9}), we have
\begin{eqnarray}\label{eq10}
\mathcal{I}_{a+}^{\alpha ;\Psi }\left( fg\right) \left( x\right) &=&\underset{n=0}{\overset{\infty }{\sum }}\left( 
\begin{array}{c}
-\alpha \\ 
n%
\end{array}%
\right) \overset{n}{\underset{k=0}{\sum }}\left( 
\begin{array}{c}
n \\ 
k%
\end{array}%
\right) f^{\left( k\right) }\left( x\right) g^{\left( n-k\right) }\left(x\right) \frac{\left( \Psi \left( x\right) -\Psi \left( a\right) \right)^{\alpha +n}}{\Gamma 
\left( \alpha +n+1\right) }  \notag \\
&=&\underset{n=0}{\overset{\infty }{\sum }}\overset{\infty }{\underset{n=k}{\sum }}\left( 
\begin{array}{c}
-\alpha \\ 
n%
\end{array}%
\right) \left( 
\begin{array}{c}
n \\ 
k%
\end{array}%
\right) f^{\left( k\right) }\left( x\right) g^{\left( n-k\right) }\left( x\right) \frac{\left( \Psi \left( x\right) -\Psi \left( a\right) \right) ^{\alpha +n}}{\Gamma 
\left( \alpha +n+1\right) }  \notag \\
&=&\underset{k=0}{\overset{\infty }{\sum }}f^{\left( k\right) }\left( x\right) \overset{\infty }{\underset{n=k}{\sum }}\left( 
\begin{array}{c}
-\alpha \\ 
n%
\end{array}%
\right) \left( 
\begin{array}{c}
n \\ 
k%
\end{array}%
\right) g^{\left( n-k\right) }\left( x\right) \frac{\left( \Psi \left( x\right) -\Psi \left( a\right) \right) ^{\alpha +n}}{\Gamma \left( \alpha +n+1\right) }\notag \\
\end{eqnarray}

Entering the index change $n\rightarrow n + k$ in the second summation on the right side of Eq.(\ref{eq10}), we can write
\begin{equation}\label{eq11}
\mathcal{I}_{a+}^{\alpha ;\Psi }\left( fg\right) \left( x\right) =\underset{k=0}{\overset{\infty }{\sum }}f^{\left( k\right) }\left( x\right) \overset{\infty }{\underset{n=0}{\sum }}\left( 
\begin{array}{c}
-\alpha \\ 
n+k%
\end{array}%
\right) \left( 
\begin{array}{c}
n+k \\ 
k%
\end{array}
\right) g^{\left( n\right) }\left( x\right) \frac{\left( \Psi \left( x\right) -\Psi \left( a\right) \right) ^{\alpha +n+k}}{\Gamma \left( \alpha +n+k+1\right) }.
\end{equation}

Note that
\begin{eqnarray}\label{eq12}
\left( 
\begin{array}{c}
-\alpha \\ 
n+k%
\end{array}%
\right) \left( 
\begin{array}{c}
n+k \\ 
k%
\end{array}%
\right) &=&\frac{\Gamma \left( 1-\alpha \right) }{\Gamma \left( 1-\alpha -n-k\right) \left( n+k\right) !}\frac{\left( n+k\right) !}{n!k!}  \notag \\
&=&\frac{\Gamma \left( 1-\alpha \right) }{\Gamma \left( 1-\alpha -n-k\right) n!k!}\frac{\Gamma \left( 1-\alpha -k\right) }{\Gamma \left( 1-\alpha -k\right) }  \notag \\
&=&\left( 
\begin{array}{c}
-\alpha \\ 
k%
\end{array}%
\right) \left( 
\begin{array}{c}
-\alpha -k \\ 
n%
\end{array}%
\right) .
\end{eqnarray}

Substituting Eq.(\ref{eq12}) into Eq.(\ref{eq11}), we get
\begin{equation*}
\mathcal{I}_{a+}^{\alpha ;\Psi }\left( fg\right) \left( x\right) =\underset{k=0}{\overset{\infty }{\sum }}f^{\left( k\right) }\left( x\right) \overset{\infty }{\underset{n=0}{\sum }}\left( 
\begin{array}{c}
-\alpha \\ 
k%
\end{array}%
\right) \left( 
\begin{array}{c}
-\alpha -k \\ 
n%
\end{array}%
\right) g^{\left( n\right) }\left( x\right) \frac{\left( \Psi \left(x\right) -\Psi \left( a\right) \right) ^{\alpha +n+k}}{\Gamma \left( \alpha +n+k+1\right) }.
\end{equation*}

By means of Eq.(\ref{eq6}), we obtain an expression for the $\Psi-$RL fractional integral of the product of two functions $f$ and $g$, given by
\begin{equation}\label{fu}
\mathcal{I}_{a+}^{\alpha ;\Psi }\left( fg\right) \left( x\right) =\underset{k=0}{\overset{\infty }{\sum }}\left( 
\begin{array}{c}
-\alpha \\ 
k
\end{array}
\right) f^{\left( k\right) }\left( x\right) \mathcal{I}_{a+}^{\alpha +k;\Psi }g\left(x\right),
\end{equation}
which is the desired result
\end{proof}

\section{Leibniz type rule}

Recently Sousa and Oliveira introduced the $\Psi-$H fractional derivative as presented in Eq.(\ref{eq3}) that contemplates a wide class of fractional derivatives.
Moreover, in this work we can find results that guarantee that the operator is well defined, as well as other relevant results that contribute in the field
of the fractional calculation. However, the paper did not address the Leibniz rule. In this sense, we will obtain a formula for the $\Psi-$H fractional 
derivative of the product of two functions, the so-called Leibniz rule of the function $\Psi$, of type I and type II, that is, one as a function of the 
$\Psi-$RL fractional derivative and the other as a function of the $\Psi-$H fractional derivative.

By means of the expression of the $\Psi-$RL fractional derivative of the product of two functions, as obtained in the previous section, we will present two 
expressions for the Leibniz rule with respect to another function, specifically, we present some particular cases, this is, possible formulations of Leibniz or Leibniz type rules.

For the investigation of the Leibniz rule of the $\Psi-$C fractional derivative given by Eq.(\ref{eq5}) it is sufficient to consider $n = 1$.

\begin{theorem}{\rm (Leibniz Type Rule I).} Let $0<\alpha <1$ and $\Lambda =\left[ a,b\right] $ an interval such that $-\infty \leq a<b\leq \infty $, $\Psi $ a growing function 
such that $\Psi ^{\prime }\left( x\right) \neq 0$, $t\in \Lambda $ and $\Psi \in C\left( \Lambda ,\mathbb{R}\right)$, $f,g\in C\left( \Lambda ,\mathbb{R}\right)$. Then, we have
\begin{eqnarray}\label{eq13}
^{\mathbf{H}}\mathfrak{D}_{a+}^{\alpha ,\beta ;\Psi }\left( fg\right) \left( x\right) &=&\underset{l=0}{\overset{\infty }{\sum }}\overset{\infty }{\underset{m=0}{\sum }}\left( 
\begin{array}{c}
-\xi \\ 
m-l
\end{array}%
\right) \left( 
\begin{array}{c}
\alpha-\xi \\ 
l%
\end{array}%
\right) f^{\left( n\right) }\left( x\right) \text{ }^{\mathbf{RL}}\mathfrak{D}_{a+}^{\alpha -m;\Psi }g\left( x\right)  \notag \\ &&-\underset{k=0}{\overset{\infty }{\sum }}\left( 
\begin{array}{c}
-\xi \\ 
k%
\end{array}%
\right) \mathcal{I}_{a+}^{\xi +k;\Psi}g\left( a\right) f^{\left( k\right) }\left( a\right) \frac{\left( \Psi \left( x\right) -\Psi \left( a\right) \right) ^{\xi-\alpha }}{\Gamma 
\left( \beta \left( 1-\alpha \right) \right) }.\notag \\
\end{eqnarray}
\end{theorem}

\begin{proof}
For the proof of Eq.(\ref{eq13}), let us use the relation between the $\Psi-$H fractional derivative and $\Psi-$C fractional derivative, given by
\begin{equation}\label{eq14}
^{\mathbf{H}}\mathfrak{D}_{a+}^{\alpha ,\beta ;\Psi }f\left( x\right) =\text{ }^{\mathbf{C}}\mathfrak{D}_{a+}^{\beta\left( \alpha -n\right) +1;\Psi }\mathcal{I}_{a+}^{\xi ;\Psi }f
\left( x\right),
\end{equation}
with $n-1<\alpha\leq n$ and $0\leq \beta \leq 1$.

Thus, taking the product of $f$ and $g$ and substituting in Eq.(\ref{eq14}), we have
\begin{equation}\label{eq15}
^{\mathbf{H}}\mathfrak{D}_{a+}^{\alpha ,\beta ;\Psi }\left( fg\right) \left( x\right) =\text{ }
^{\mathbf{C}}\mathfrak{D}_{a+}^{\beta \left( \alpha -n\right) +1;\Psi }\mathcal{I}_{a+}^{\xi ;\Psi }\left( fg\right) \left( x\right) 
\end{equation}

Substituting Eq.(\ref{fu}) into Eq.(\ref{eq15}), we get
\begin{eqnarray}\label{eq17}
^{\mathbf{H}}\mathfrak{D}_{a+}^{\alpha ,\beta ;\Psi }\left( fg\right) \left( x\right)
&=&\text{ }^{\mathbf{C}}\mathfrak{D}_{a+}^{\alpha-\xi;\Psi }\overset{\infty }{\underset{k=0}{\sum }}\left( 
\begin{array}{c}
-\xi \\ 
k
\end{array}
\right) f^{\left( k\right) }\left( x\right) \mathcal{I}_{a+}^{\xi +k;\Psi }g\left( x\right)  \notag \\ &=&\overset{\infty }{\underset{k=0}{\sum }}\left( 
\begin{array}{c}
-\xi \\ 
k
\end{array}
\right) \text{ }^{\mathbf{C}}\mathfrak{D}_{a+}^{\alpha-\xi;\Psi }\left(f^{\left( k\right) }\left( x\right) \mathcal{I}_{a+}^{\xi +k;\Psi }g\left( x\right) \right).\notag \\
\end{eqnarray}

The generalization of the Leibniz rule for the $\Psi-$C fractional derivative in terms of the $\Psi-$RL fractional derivative is given by
\begin{eqnarray}\label{eq18}
^{\mathbf{C}}\mathfrak{D}_{a+}^{\alpha ;\Psi }\left( fg\right) \left( x\right) &=&\overset{\infty }{\underset{k=0}{\sum }}\left( 
\begin{array}{c}
\alpha \\ 
k%
\end{array}%
\right) f^{\left( k\right) }\left( x\right) \text{ }^{\mathbf{RL}}\mathfrak{D}_{a+}^{\alpha -k;\Psi }g\left( x\right)  \notag \\&&-\underset{k=0}{\overset{n-1}{\sum }}
\frac{\dfrac{d^{k}}{dx^{k}}\left( f\left( x\right) g\left( x\right) \right) \left( a\right) }{\Gamma \left( k-\alpha +1\right) }\left( \Psi \left( x\right) -\Psi 
\left( a\right)\right) ^{k-\alpha }.
\end{eqnarray}

Using the Leibniz rule for the $\Psi-$C fractional derivative, substituting Eq.(\ref{eq17}) into Eq.(\ref{eq18}), we have
\begin{eqnarray}\label{eq19}
^{\mathbf{H}}\mathfrak{D}_{a+}^{\alpha ,\beta ;\Psi }\left( fg\right) \left( x\right) &=&\overset{\infty }{\underset{k=0}{\sum }}\left( 
\begin{array}{c}
-\xi \\ 
k%
\end{array}%
\right) \overset{\infty }{\underset{l=0}{\sum }}\left( 
\begin{array}{c}
\alpha-\xi \\ 
l%
\end{array}%
\right) f^{\left( k+l\right) }\left( x\right) \text{ }^{\mathbf{RL}}\mathfrak{D}_{a+}^{\alpha-\xi-l;\Psi }\mathcal{I}_{a+}^{\xi +k;\Psi }g\left( x\right)  \notag \\
&&-\overset{\infty }{\underset{k=0}{\sum }}\left( 
\begin{array}{c}
-\xi \\ 
k%
\end{array}%
\right) \mathcal{I}_{a+}^{\xi +k;\Psi }g\left( a\right) f^{\left( k\right) }\left( a\right) \frac{\left( \Psi \left( x\right) -\Psi \left( a\right) \right) ^{\xi-\alpha }}{\Gamma 
\left( \beta \left( 1-\alpha \right) \right) }.
\end{eqnarray}

Considering the index change $m \rightarrow k + l$ in the first sum of Eq.(\ref{eq19}), we have
\begin{eqnarray}\label{eq20}
^{\mathbf{H}}\mathfrak{D}^{\alpha ,\beta ;\Psi }\left( fg\right) \left( x\right)
&=&\overset{\infty }{\underset{l=0}{\sum }}\overset{\infty }{\underset{m=0}{\sum }}\left( 
\begin{array}{c}
-\xi \\ 
m-l%
\end{array}%
\right) \left( 
\begin{array}{c}
\alpha-\xi \\ 
l%
\end{array}%
\right) f^{\left( m\right) }\left( x\right) \text{ }^{\mathbf{RL}}\mathfrak{D}^{\alpha-\xi-l;\Psi }\mathcal{I}_{a+}^{\xi +m-l;\Psi }g\left( x\right)  \notag \\
&&-\overset{\infty }{\underset{k=0}{\sum }}\left( 
\begin{array}{c}
-\xi \\ 
k%
\end{array}%
\right) \mathcal{I}_{a+}^{\xi +k;\Psi }g\left( a\right) f^{\left( k\right) }\left( a\right) \frac{\left( \Psi \left( x\right) -\Psi \left( a\right) \right) ^{\xi-\alpha }}{\Gamma 
\left( \beta \left( 1-\alpha \right) \right) }.
\end{eqnarray}

Note that
\begin{eqnarray*}
^{\mathbf{RL}}\mathfrak{D}_{a+}^{\gamma -l;\Psi }\mathcal{I}_{a+}^{\mu -l;\Psi }h\left( x\right) &=&\left( \frac{1}{\Psi ^{\prime }\left( x\right) }\frac{d}{dx}\right) 
\mathcal{I}_{a+}^{1-\gamma +l;\Psi }\mathcal{I}_{a+}^{\mu -l;\Psi }h\left( x\right)  \notag \\&=&\left( \frac{1}{\Psi ^{\prime }\left( x\right) }\frac{d}{dx}\right) 
\mathcal{I}_{a+}^{1-\gamma +\mu ;\Psi }h\left( x\right)  \notag \\
&=&\text{ }^{\mathbf{RL}}\mathfrak{D}_{a+}^{\alpha -\beta ;\Psi }h\left( x\right) .
\end{eqnarray*}

Thus, considering $\gamma =\alpha-\xi$ and $\mu =\xi +m$, we get
\begin{equation}\label{eq21}
^{\mathbf{RL}}\mathfrak{D}_{a+}^{\alpha-\xi-l;\Psi }\mathcal{I}_{a+}^{\xi +m-l;\Psi }h\left( x\right) =\text{ }^{\mathbf{RL}}\mathfrak{D}_{a+}^{\alpha -m;\Psi }h\left( x\right) .
\end{equation}

Substituting Eq.(\ref{eq21}) into Eq.(\ref{eq20}), we conclude 
\begin{eqnarray*}
^{\mathbf{H}}\mathfrak{D}_{a+}^{\alpha ,\beta ;\Psi }\left( fg\right) \left( x\right) &=&\overset{\infty }{\underset{l=0}{\sum }}\overset{\infty }{\underset{m=0}{%
\sum }}\left( 
\begin{array}{c}
-\xi \\ 
m-l%
\end{array}%
\right) \left( 
\begin{array}{c}
\alpha-\xi \\ 
l%
\end{array}%
\right) f^{\left( m\right) }\left( x\right) \text{ }^{\mathbf{RL}}\mathfrak{D}_{a+}^{\alpha
-m;\Psi }g\left( x\right)  \notag \\
&&-\overset{\infty }{\underset{k=0}{\sum }}\left( 
\begin{array}{c}
-\xi \\ 
k%
\end{array}%
\right) \mathcal{I}_{a+}^{\xi +k;\Psi }g\left( a\right) f^{\left( k\right) }\left( a\right) \frac{\left( \Psi \left( x\right) -\Psi \left( a\right) \right) ^{\xi-\alpha }}{\Gamma 
\left( \beta \left( 1-\alpha \right) \right) },
\end{eqnarray*}
which is the desired result.
\end{proof}

Here we present the Leibniz type rule in terms of the $\Psi-$H fractional derivative and from this formulation we recover some important particular cases 
for the context of fractional derivatives.

\begin{theorem} {\rm (Leibniz Type Rule II).} Let $0<\alpha <1$ and $\Lambda =\left[ a,b\right]$ an interval such that $-\infty \leq a<b\leq \infty $, $\Psi $ a function
increasing such that $\Psi ^{\prime }\left( x\right) \neq 0$, $t\in \Lambda $ and $\Psi \in C\left( \Lambda ,\mathbb{R}\right) $, $f,g\in C\left( \Lambda ,\mathbb{R}\right) $. 
Then, we have
\begin{eqnarray}\label{eq22}
^{\mathbf{H}}\mathfrak{D}_{a+}^{\alpha ,\beta ;\Psi }\left( fg\right) \left( x\right) &=&\overset{\infty }{\underset{m=0}{\sum }}\left( 
\begin{array}{c}
\alpha  \\ 
m
\end{array}%
\right) f^{\left( m\right) }\left( x\right) \text{ }^{\mathbf{H}}\mathfrak{D}_{a+}^{\alpha -m,\beta ;\Psi }g\left( x\right) + \Omega_{f,g}(\alpha,\beta,a) \notag \\
\end{eqnarray}
with 
$$
\Omega_{f,g}(\alpha,\beta,a)=\underset{k=0}{\overset{\infty }{\sum }}\left( \begin{array}{c} -\xi  \\ k\end{array} \right) \mathcal{I}_{a+}^{\xi +k;\Psi }g
\left( a\right) \left( f^{\left( k\right) }\left( x\right) -f^{\left( k\right) }\left( a\right) \right) \dfrac{\left( \Psi \left( x\right) -\Psi \left( a\right) 
\right)^{\xi-\alpha}}{\Gamma \left( \beta -\alpha \beta \right) }.
$$
\end{theorem}

\begin{proof} For the proof of Eq.(\ref{eq22}), let us use the relation between the $\Psi-$H fractional derivative and $\Psi-$C fractional derivative, given by
\begin{equation}\label{eq23}
^{\mathbf{H}}\mathfrak{D}_{a+}^{\alpha ,\beta ;\Psi }f\left( x\right) =\text{ }^{\mathbf{C}}\mathfrak{D}_{a+}^{\beta
\left( \alpha -n\right) +1;\Psi }\mathcal{I}_{a+}^{\xi ;\Psi }f\left( x\right).
\end{equation}

Thus, taking the product of $f$ and $g$ and substituting in Eq.(\ref{eq23}), we have
\begin{equation}\label{eq24}
^{\mathbf{H}}\mathfrak{D}_{a+}^{\alpha ,\beta ;\Psi }\left( fg\right) \left( x\right) =\text{ }^{\mathbf{C}}\mathfrak{D}_{a+}^{\beta \left( \alpha -n\right) +1;\Psi }
\mathcal{I}_{a+}^{\xi ;\Psi }\left( fg\right) \left( x\right).
\end{equation}

Introducing Eq.(\ref{eq13}) into Eq.(\ref{eq24}), it follows that
\begin{eqnarray}\label{eq25}
^{\mathbf{H}}\mathfrak{D}_{a+}^{\alpha ,\beta ;\Psi }\left( fg\right) \left( x\right)  &=&\text{ }^{\mathbf{C}}\mathfrak{D}_{a+}^{\alpha-\xi;\Psi }\overset{\infty }{\underset{k=0}{\sum }}\left( 
\begin{array}{c}
-\xi \\ 
k%
\end{array}%
\right) f^{\left( k\right) }\left( x\right) \mathcal{I}_{a+}^{\xi +k;\Psi }g\left( x\right)  \notag \\
&=&\overset{\infty }{\underset{k=0}{\sum }}\left( 
\begin{array}{c}
-\xi \\ 
k%
\end{array}%
\right) \text{ }^{\mathbf{C}}\mathfrak{D}_{a+}^{\alpha-\xi;\Psi }\left( f^{\left( k\right) }\left( x\right) \mathcal{I}_{a+}^{\xi +k;\Psi }g\left( x\right) \right).\notag \\
\end{eqnarray}

The generalization of the Leibniz rule to the $\Psi-$C fractional derivative in terms of the $\Psi-$C fractional derivatives, is given by
\begin{equation}\label{eq26}
^{C}\mathfrak{D}_{a+}^{\alpha ;\Psi }\left( fg\right) \left( x\right) =\overset{\infty }{\underset{k=0}{\sum }}\left( 
\begin{array}{c}
\alpha \\ 
k%
\end{array}%
\right) f^{\left( k\right) }\left( x\right) \text{ }^{\mathbf{C}}\mathfrak{D}_{a+}^{\alpha -k;\Psi }g\left( x\right) +g\left( a\right) \left( f\left( t\right) -f
\left( a\right) \right) \frac{\left( \Psi \left( x\right) -\Psi \left( a\right) \right) ^{-\alpha }}{\Gamma \left( 1-\alpha \right) }.
\end{equation}

Substituting Eq.(\ref{eq26}) into Eq.(\ref{eq25}), we get
\newpage
\begin{eqnarray*}
^{\mathbf{H}}\mathfrak{D}_{a+}^{\alpha ,\beta ;\Psi }\left( fg\right) \left( x\right) 
&=&\text{ }\overset{\infty }{\underset{k=0}{\sum }}\left( 
\begin{array}{c}
-\xi  \\ 
k%
\end{array}%
\right) \left\{ \underset{s=0}{\overset{\infty }{\sum }}\left( 
\begin{array}{c}
\alpha-\xi \\ 
s%
\end{array}%
\right) f^{\left( k+s\right) }\left( x\right) \text{ }^{\mathbf{C}}\mathfrak{D}_{a+}^{\alpha-\xi-s;\Psi }\mathcal{I}_{a+}^{\xi +k;\Psi }g\left( x\right) \right\}   \notag \\
&&+\overset{\infty }{\underset{k=0}{\sum }}\left( 
\begin{array}{c}
-\xi  \\ 
k%
\end{array}%
\right) \text{ }\mathcal{I}_{a+}^{\xi +k;\Psi }g\left( a\right) \left( f^{\left( k\right) }\left( x\right) -f^{\left( k\right) }\left( a\right) \right) 
\frac{\left( \Psi \left( x\right) -\Psi \left( a\right) \right) ^{\xi-\alpha}}{\Gamma \left( \beta -\alpha \beta \right) }.
\end{eqnarray*}

Considering the change of index $k=m-s$, we have
\begin{eqnarray}\label{eq27}
^{\mathbf{H}}\mathfrak{D}_{a+}^{\alpha ,\beta ;\Psi }\left( fg\right) \left( x\right) 
&=&\text{ }\overset{\infty }{\underset{s=0}{\sum }}\underset{m=0}{\overset{%
\infty }{\sum }}\left( 
\begin{array}{c}
-\xi  \\ 
m-s%
\end{array}%
\right) \left( 
\begin{array}{c}
\alpha-\xi \\ 
s%
\end{array}%
\right) f^{\left( m\right) }\left( x\right) \text{ }^{\mathbf{C}}\mathfrak{D}_{a+}^{\alpha-\xi-s;\Psi }\mathcal{I}_{a+}^{\xi +m-s;\Psi }g\left( x\right)   \notag \\
&&+\overset{\infty }{\underset{k=0}{\sum }}\left( 
\begin{array}{c}
-\xi  \\ 
k%
\end{array}%
\right) \text{ }\mathcal{I}_{a+}^{\xi +k;\Psi }g\left( a\right) \left( f^{\left( k\right) }\left( x\right) -f^{\left( k\right) }\left( a\right) \right) \frac{\left( \Psi \left(
x\right) -\Psi \left( a\right) \right) ^{\xi-\alpha}}{\Gamma \left( \beta -\alpha \beta \right) }.
\end{eqnarray}

Note that, for $0<\overline{\alpha }<1,$ $0<\overline{\beta }<1$ and $s\in\mathbb{N}$ we have the identity
\begin{eqnarray}\label{eq28}
^{\mathbf{C}}\mathfrak{D}_{a+}^{\overline{\alpha }-s;\Psi }\mathcal{I}_{a+}^{\overline{\beta }-s;\Psi }\left( \cdot \right) &=&\mathcal{I}_{a+}^{1-s-
\left( \overline{\alpha }-s\right) ;\Psi }D^{1-s;\Psi }\mathcal{I}_{a+}^{\overline{\beta }-s;\Psi }\left( \cdot \right) 
\notag \\
&=&\mathcal{I}_{a+}^{1-\overline{\alpha };\Psi }D^{1;\Psi }D^{-s;\Psi }\mathcal{I}_{a+}^{-s;\Psi }\mathcal{I}_{a+}^{\overline{\beta };\Psi }\left( \cdot \right)  \notag \\ 
&=&\mathcal{I}_{a+}^{1-\overline{\alpha };\Psi }D^{1;\Psi }\mathcal{I}_{a+}^{\overline{\beta };\Psi }\left( \cdot \right).
\end{eqnarray}

Taking $\overline{\alpha }=\alpha-\xi$ e $\overline{\beta }=m+\xi $ in Eq.(\ref{eq28}), we obtain
\begin{eqnarray}\label{eq29}
^{\mathbf{C}}\mathfrak{D}_{a+}^{\alpha-\xi-s;\Psi }I_{a+}^{m+\xi ;\Psi }\left( \cdot \right) &=&\mathcal{I}_{a+}^{\beta \left( 1-\alpha \right) ;\Psi }D^{1;\Psi }
\mathcal{I}_{a+}^{\xi ;\Psi }\mathcal{I}_{a+}^{m;\Psi }\left( \cdot \right)  \notag \\ 
&=&\text{ }^{\mathbf{H}}\mathfrak{D}_{a+}^{\alpha ,\beta ;\Psi }\mathcal{I}_{a+}^{m;\Psi }\left( \cdot \right)  \notag \\
&=&\text{ }^{\mathbf{H}}\mathfrak{D}_{a+}^{\alpha -m,\beta ;\Psi }\left( \cdot \right).
\end{eqnarray}

Substituting Eq.(\ref{eq29}) into Eq.(\ref{eq27}), it follows
\begin{eqnarray*}
^{\mathbf{H}}\mathfrak{D}_{a+}^{\alpha ,\beta ;\Psi }\left( fg\right) \left( x\right) 
&=&\text{ }\overset{\infty }{\underset{s=0}{\sum }}\underset{m=0}{\overset{%
\infty }{\sum }}\left( 
\begin{array}{c}
-\xi  \\ 
m-s%
\end{array}
\right) \left( 
\begin{array}{c}
\alpha-\xi \\ 
s%
\end{array}%
\right) f^{\left( m\right) }\left( x\right) \text{ }^{\mathbf{H}}\mathfrak{D}_{a+}^{\alpha -m,\beta ;\Psi }g\left( x\right)   \notag \\
&&+\overset{\infty }{\underset{k=0}{\sum }}\left( 
\begin{array}{c}
-\xi  \\ 
k
\end{array}
\right) \text{ }\mathcal{I}_{a+}^{\xi +k;\Psi }g\left( a\right) \left( f^{\left( k\right) }\left( x\right) -f^{\left( k\right) }\left( a\right) \right) 
\frac{\left( \Psi \left( x\right) -\Psi \left( a\right) \right) ^{\xi-\alpha}}{\Gamma \left( \beta -\alpha \beta \right) }.
\end{eqnarray*}

Using the relation
\begin{equation*}
\overset{\infty }{\underset{n=0}{\sum }}\left( 
\begin{array}{c}
\alpha  \\ 
m-n%
\end{array}%
\right) \left( 
\begin{array}{c}
\beta  \\ 
n%
\end{array}%
\right) =\left( 
\begin{array}{c}
\alpha +\beta  \\ 
n%
\end{array}%
\right),
\end{equation*}
we can write
\begin{eqnarray*}
^{\mathbf{H}}\mathfrak{D}_{a+}^{\alpha ,\beta ;\Psi }\left( fg\right) \left( x\right) &=&\overset{\infty }{\underset{m=0}{\sum }}\left( 
\begin{array}{c}
\alpha  \\ 
m
\end{array}%
\right) f^{\left( m\right) }\left( x\right) \text{ }^{\mathbf{H}}\mathfrak{D}_{a+}^{\alpha -m,\beta ;\Psi }g\left( x\right) + \Omega_{f,g}(\alpha,\beta,a) \notag \\
\end{eqnarray*}
with 
$$
\Omega_{f,g}(\alpha,\beta,a)=\underset{k=0}{\overset{\infty }{\sum }}\left( \begin{array}{c} -\xi  \\ k\end{array} \right) \mathcal{I}_{a+}^{\xi +k;\Psi }
g\left( a\right) \left( f^{\left( k\right) }\left( x\right) -f^{\left( k\right) }\left( a\right) \right) \dfrac{\left( \Psi \left( x\right) -\Psi \left( a\right) 
\right) ^{\xi-\alpha}}{\Gamma \left( \beta -\alpha \beta \right) }
$$ 
which is the desired result.
\end{proof}

Note that the Leibniz rule for the $\Psi-$H fractional derivative is not simply
\begin{equation*}
^{\mathbf{H}}\mathfrak{D}_{a+}^{\alpha ,\beta ;\Psi }\left( fg\right) \left( x\right)=\left(^{\mathbf{H}}\mathfrak{D}_{a+}^{\alpha ,\beta ;\Psi }f\left( x\right) \right) g\left( x\right) +
\text{ }f\left( x\right) \text{ }\left( ^{\mathbf{H}}\mathfrak{D}_{a+}^{\alpha ,\beta ;\Psi}g\left( x\right) \right).
\end{equation*}

\begin{remark} One of the conditions for a given differential operator to be considered fractional as {\rm \cite{ortigueira}}, is to satisfy Leibniz's rule. 
However, some traders are not satisfied, since an extra portion appears. Many fractional derivatives in particular, the Caputo fractional derivative, one of the most used and 
essential in the study of memory effects and under differential equations with initial conditions, does not satisfy the Leibniz rule.
In this sense, since few derivatives satisfy and others do not satisfy the Leibniz rule, we introduce the Leibniz rule for the $\Psi-$H fractional derivative, 
since from this it is possible to obtain, as particular cases, a broad class of derived fractions, consequently their respective Leibniz rule and Leibniz type rule as well.
\end{remark}

In this sense, we present the Leibniz type rule in two versions for the $\Psi-$H fractional derivative, however for continuity of work, let's stop the version of the 
Leibniz type rule, according to Theorem 2.

Here are a few special cases of the Leibniz rule from the $\Psi\left(\cdot \right)$ function and the $\beta \rightarrow 0$ and $\beta \rightarrow 1$ limits. 
For the particular cases presented below, we suggest \cite{SAMKO,POD,tarasov,tarasov1,tarasov2,tarasov3,jumarie,wang,diethelm1}.

\begin{enumerate}
\item Taking the limit $\beta \rightarrow 0$ in both sides of Eq.(\ref{eq22}), we get the Leibniz rule for the $\Psi-$RL fractional derivative,
\begin{eqnarray*}
^{\mathbf{H}}\mathfrak{D}_{a+}^{\alpha ,0;\Psi }\left( fg\right) \left( x\right) &=&\overset{\infty }{\underset{m=0}{\sum }}\left( 
\begin{array}{c}
\alpha  \\ 
m%
\end{array}%
\right) f^{\left( m\right) }\left( x\right) \text{ }^{\mathbf{RL}}\mathfrak{D}_{a+}^{\alpha -m;\Psi }g\left( x\right)   \notag \\
&&+\underset{k=0}{\overset{\infty }{\sum }}\left( 
\begin{array}{c}
\alpha -1 \\ 
k%
\end{array}%
\right) \mathcal{I}_{a+}^{\left( 1-\alpha \right) +k;\Psi }g\left( a\right) \left( f^{\left( k\right) }\left( x\right) -f^{\left( k\right) }\left( a\right) \right) 
\left( \Psi \left( x\right) -\Psi \left( a\right) \right) ^{-1} 
\notag \\
&=&\overset{\infty }{\underset{m=0}{\sum }}\left( 
\begin{array}{c}
\alpha  \\ 
m%
\end{array}%
\right) f^{\left( m\right) }\left( x\right) \text{ }^{\mathbf{RL}}\mathfrak{D}_{a+}^{\alpha -m;\Psi }g\left( x\right) =\text{ }^{\mathbf{RL}}\mathfrak{D}_{a+}^{\alpha ;\Psi }
\left( fg\right) \left( x\right) .
\end{eqnarray*}

\item Taking $\Psi \left( x\right) =x$ and considering the limit $\beta \rightarrow 0$ in both sides of Eq.(\ref{eq22}), we get the Leibniz rule for the Riemann-Liouville 
fractional derivative,
\begin{eqnarray*}
^{\mathbf{H}}\mathfrak{D}_{a+}^{\alpha ,0;x}\left( fg\right) \left( x\right)  &=&\overset{\infty }{\underset{m=0}{\sum }}\left( 
\begin{array}{c}
\alpha  \\ 
m%
\end{array}%
\right) f^{\left( m\right) }\left( x\right) \text{ }^{\mathbf{RL}}\mathfrak{D}^{\alpha -m}g\left( x\right)   \notag \\
&&+\underset{k=0}{\overset{\infty }{\sum }}\left( 
\begin{array}{c}
\alpha -1 \\ 
k%
\end{array}%
\right) \mathcal{I}_{a+}^{\left( 1-\alpha \right) +k}g\left( a\right) \left( f^{\left( k\right) }\left( x\right) -f^{\left( k\right) }\left( a\right) \right) \left( x-a\right) ^{-1}  \notag \\
&=&\overset{\infty }{\underset{m=0}{\sum }}\left( 
\begin{array}{c}
\alpha  \\ 
m%
\end{array}%
\right) f^{\left( m\right) }\left( x\right) \text{ }^{\mathbf{RL}}\mathfrak{D}_{a+}^{\alpha -m}g\left( x\right) =\text{ }^{\mathbf{RL}}\mathfrak{D}_{a+}^{\alpha }\left( fg\right) \left(x\right) .
\end{eqnarray*}

\item Taking the limit $\beta \rightarrow 1$ in both sides of Eq.(\ref{eq22}), we get the Leibniz rule for the $\Psi-$C fractional derivative
\begin{eqnarray*}
^{\mathbf{H}}\mathfrak{D}_{a+}^{\alpha ,1;\Psi }\left( fg\right) \left( x\right)  &=&\overset{\infty }{\underset{m=0}{\sum }}\left( 
\begin{array}{c}
\alpha  \\ 
m%
\end{array}%
\right) f^{\left( m\right) }\left( x\right) \text{ }^{\mathbf{C}}\mathfrak{D}_{a+}^{\alpha -m;\Psi }g\left( x\right)   \notag \\&&+\underset{k=0}{\overset{\infty }{\sum }}\left( 
\begin{array}{c}
0 \\ 
k%
\end{array}%
\right) \mathcal{I}_{a+}^{k;\Psi }g\left( a\right) \left( f^{\left( k\right) }\left( x\right) -f^{\left( k\right) }\left( a\right) \right) \frac{\left( \Psi 
\left( x\right) -\Psi \left( a\right) \right) }{\Gamma \left( 1-\alpha \right) }^{-\alpha }  \notag \\ &=&\overset{\infty }{\underset{m=0}{\sum }}\left( 
\begin{array}{c}
\alpha  \\ 
m%
\end{array}%
\right) f^{\left( m\right) }\left( x\right) \text{ }^{\mathbf{C}}\mathfrak{D}_{a+}^{\alpha -m;\Psi }g\left( x\right) +\Omega^{1}_{f,g}(\alpha,a) \notag \\ 
&=&\text{ }^{\mathbf{C}}\mathfrak{D}_{a+}^{\alpha ;\Psi }\left( fg\right) \left( x\right),
\end{eqnarray*}
with 
$$
\Omega_{f,g}(\alpha,a)= g\left( a\right) \left( f^{\left( k\right) }\left( x\right) -f^{\left( k\right) }\left( a\right) \right) \dfrac{\left( \Psi 
\left( x\right) -\Psi \left( a\right) \right) }{\Gamma \left( 1-\alpha \right) }^{-\alpha }.
$$

\item Considering $\Psi \left( x\right) =x$ and taking the limit $\beta \rightarrow 1$ in both sides of Eq.(\ref{eq22}), we obtain the Leibniz rule for the Caputo fractional derivative
\begin{eqnarray*}
^{\mathbf{H}}\mathfrak{D}_{a+}^{\alpha ,1;x}\left( fg\right) \left( x\right)  &=&\overset{\infty }{\underset{m=0}{\sum }}\left( 
\begin{array}{c}
\alpha  \\ 
m%
\end{array}%
\right) f^{\left( m\right) }\left( x\right) \text{ }^{\mathbf{C}}\mathfrak{D}_{a+}^{\alpha -m}g\left( x\right)   \notag \\&&+\underset{k=0}{\overset{\infty }{\sum }}\left( 
\begin{array}{c}
0 \\ 
k%
\end{array}%
\right) I_{a+}^{k}g\left( a\right) \left( f^{\left( k\right) }\left( x\right) -f^{\left( k\right) }\left( a\right) \right) \frac{\left( x-a\right) ^{-\alpha }}
{\Gamma \left( 1-\alpha \right) }  \notag \\
&=&\overset{\infty }{\underset{m=0}{\sum }}\left( 
\begin{array}{c}
\alpha  \\ 
m%
\end{array}%
\right) f^{\left( m\right) }\left( x\right) \text{ }^{\mathbf{C}}\mathfrak{D}^{\alpha -m}g\left( x\right) +\Omega^{2}_{f,g}(\alpha,a) \notag \\ 
&=&\text{ }^{\mathbf{C}}\mathfrak{D}_{a+}^{\alpha }\left( fg\right) \left( x\right),
\end{eqnarray*}
with
$$\Omega^{2}_{f,g}(\alpha,a)=g\left( a\right) \left( f^{\left( k\right) }\left( x\right) -f^{\left( k\right) }\left( a\right) \right) 
\dfrac{\left(x-a\right) ^{-\alpha }}{\Gamma \left( 1-\alpha \right) }.
$$

\item Considering $\Psi \left( x\right) =x$ and substituting in Eq.(\ref{eq22}), we get the Leibniz rule for the Hilfer fractional derivative 
\begin{eqnarray*}
^{\mathbf{H}}\mathfrak{D}_{a+}^{\alpha ,\beta ;x}\left( fg\right) \left( x\right)&=&\overset{\infty }{\underset{m=0}{\sum }}\left( 
\begin{array}{c}
\alpha  \\ 
m%
\end{array}
\right) f^{\left( m\right) }\left( x\right) \text{ }^{\mathbf{H}}\mathfrak{D}_{a+}^{\alpha -m,\beta }g\left( x\right) +\Omega^{3}_{f,g}(\alpha,\beta,a) \\
&=&\text{ }^{\mathbf{H}}\mathfrak{D}_{a+}^{\alpha ,\beta }\left( fg\right) \left( x\right),
\end{eqnarray*}
with 
$$
\Omega^{3}_{f,g}(\alpha,\beta,a)= \underset{k=0}{\overset{\infty }{\sum }}\left( 
\begin{array}{c} -\xi  \\ k\end{array}\right) \mathcal{I}_{a+}^{\xi +k}g\left( a\right) \left( f^{\left( k\right) }\left( x\right) -f^{\left( k\right) }
\left( a\right) \right) \dfrac{\left( x-a\right) ^{\xi-\alpha}}{\Gamma \left( \beta -\alpha \beta \right) }.
$$

\item Considering $\Psi \left( x\right) =\ln x$ and taking the limit $\beta \rightarrow 1$ in both sides of Eq.(\ref{eq22}), we obtain the Leibniz rule for the Hadamard 
fractional derivative
\begin{eqnarray*}
^{\mathbf{H}}\mathfrak{D}_{a+}^{\alpha ,1;\ln x}\left( fg\right) \left( x\right) &=&\overset{\infty }{\underset{m=0}{\sum }}\left( 
\begin{array}{c}
\alpha  \\ 
m%
\end{array}%
\right) f^{\left( m\right) }\left( x\right) \text{ }^{\mathbf{HD}}\mathfrak{D}_{a+}^{\alpha -m}g\left( x\right) + \Omega^{4}_{f,g}(\alpha,a)  \\
&=&\text{ }^{\mathbf{HD}}\mathfrak{D}_{a+}^{\alpha }\left( fg\right) \left( x\right),
\end{eqnarray*}
with 
$$
\Omega^{4}_{f,g}(\alpha,a)=g\left( a\right) \left( f^{\left( k\right) }\left( x\right) -f^{\left( k\right) }\left( a\right) \right) 
\dfrac{\left( \ln x-\ln a\right) ^{-\alpha }}{\Gamma \left( 1-\alpha \right) }.
$$

\item Considering $\Psi \left( x\right) =x^{\rho }$ and taking the limit $\beta \rightarrow 0$ in both sides of Eq.(\ref{eq22}), we get a Leibniz type rule for the 
Katugampola fractional derivative
\begin{eqnarray}
^{\mathbf{H}}\mathfrak{D}_{a+}^{\alpha ,0;x^{\rho }}\left( fg\right) \left( x\right) 
&=&\overset{\infty }{\underset{m=0}{\sum }}\left( 
\begin{array}{c}
\alpha  \\ 
m%
\end{array}%
\right) f^{\left( m\right) }\left( x\right) \text{ }\frac{1}{\rho ^{\alpha }}\text{ }^{\mathbf{K}}\mathfrak{D}_{a+}^{\alpha -m}g\left( x\right)   \notag \\
&&+\underset{k=0}{\overset{\infty }{\sum }}\left( 
\begin{array}{c}
\alpha -1 \\ 
k%
\end{array}%
\right) \mathcal{I}_{a+}^{\left( 1-\alpha \right) +k}g\left( a\right) \left( f^{\left( k\right) }\left( x\right) -f^{\left( k\right) }\left( a\right) \right) 
\left( x^{\rho }-a^{\rho }\right) ^{-1}  \notag \\
&=&\overset{\infty }{\underset{m=0}{\sum }}\left( 
\begin{array}{c}
\alpha  \\ 
m%
\end{array}%
\right) f^{\left( m\right) }\left( x\right) \text{ }\frac{1}{\rho ^{\alpha }}\text{ }^{\mathbf{K}}\mathfrak{D}^{\alpha -m}g\left( x\right) =
\text{ }\frac{1}{\rho ^{\alpha }}\text{ }^{\mathbf{K}}\mathfrak{D}_{a+}^{\alpha }\left( fg\right) \left( x\right) .
\end{eqnarray}

\item Considering $\Psi \left( x\right) =\ln x$ and taking the limit $\beta \rightarrow 1$ in both sides of Eq.(\ref{eq22}), we get a Leibniz type rule for the 
Caputo-Hadamard fractional derivative
\begin{eqnarray}
^{\mathbf{H}}\mathfrak{D}_{a+}^{\alpha ,1;\ln x}\left( fg\right) \left( x\right)  &=&\overset{\infty }{\underset{m=0}{\sum }}\left( 
\begin{array}{c}
\alpha  \\ 
m%
\end{array}%
\right) f^{\left( m\right) }\left( x\right) \text{ }^{\mathbf{CH}}\mathfrak{D}^{\alpha
-m}g\left( x\right)   \notag \\
&&+\underset{k=0}{\overset{\infty }{\sum }}\left( 
\begin{array}{c}
0 \\ 
k%
\end{array}%
\right) \mathcal{I}_{a+}^{k}g\left( a\right) \left( f^{\left( k\right) }\left( x\right) -f^{\left( k\right) }\left( a\right) \right) 
\frac{\left( \ln x-\ln a\right) ^{-\alpha }}{\Gamma \left( 1-\alpha \right) }  \notag \\
&=&\overset{\infty }{\underset{m=0}{\sum }}\left( 
\begin{array}{c}
\alpha  \\ 
m%
\end{array}%
\right) f^{\left( m\right) }\left( x\right)\text{ } ^{\mathbf{CH}}\mathfrak{D}_{a+}^{\alpha -m}g\left( x\right) + \Omega^{5}_{f,g}(\alpha,a) \notag \\ 
&=&\text{ }^{\mathbf{CH}}\mathfrak{D}_{a+}^{\alpha }\left( fg\right) \left( x\right),
\end{eqnarray}
with 
$$
\Omega^{5}_{f,g}(\alpha,a)=g\left( a\right) \left( f^{\left( k\right) }\left( x\right) -f^{\left( k\right) }\left( a\right) \right) 
\dfrac{\left( \ln x-\ln a\right) ^{-\alpha }}{\Gamma \left( 1-\alpha \right) }.
$$

\item Considering $\Psi \left( x\right) =x^{\rho }$ and taking the limit $\beta \rightarrow 1$ in both sides of Eq.(\ref{eq22}), we obtain a Leibniz type rule 
for the Caputo-Katugampola fractional derivative
\begin{eqnarray}
^{\mathbf{H}}\mathfrak{D}_{a+}^{\alpha ,1;x^{\rho }}\left( fg\right) \left( x\right)  &=&\overset{\infty }{\underset{m=0}{\sum }}\left( 
\begin{array}{c}
\alpha  \\ 
m%
\end{array}%
\right) f^{\left( m\right) }\left( x\right) \text{ }\frac{1}{\rho ^{\alpha }}\text{ }^{\mathbf{CK}}\mathfrak{D}_{a+}^{\alpha -m}g\left( x\right)   \notag \\ 
&&+\underset{k=0}{\overset{\infty }{\sum }}\left( 
\begin{array}{c}
0 \\ 
k%
\end{array}%
\right) \mathcal{I}_{a+}^{k}g\left( a\right) \left( f^{\left( k\right) }\left( x\right) -f^{\left( k\right) }\left( a\right) \right) 
\frac{\left( x^{\rho }-a^{\rho }\right) ^{-\alpha }}{\Gamma \left( 1-\alpha \right) }  \notag \\ &=&\overset{\infty }{\underset{m=0}{\sum }}\left( 
\begin{array}{c}
\alpha  \\ 
m%
\end{array}%
\right) f^{\left( m\right) }\left( x\right)\text{ } ^{\mathbf{CK}}\mathfrak{D}^{\alpha -m}g\left( x\right) +\Omega^{6}_{f,g}(\alpha,\rho,a)  \notag \\ 
&=&\frac{1}{\rho ^{\alpha }}\text{ }^{\mathbf{CK}}\mathfrak{D}_{a+}^{\alpha }\left( fg\right) \left( x\right),
\end{eqnarray}
with 
$$
\Omega^{6}_{f,g}(\alpha,\rho,a)=g\left( a\right) \left( f^{\left( k\right) }\left( x\right) -f^{\left( k\right) }\left( a\right) \right) 
\dfrac{\left( x^{\rho }-a^{\rho }\right) ^{-\alpha }}{\Gamma \left( 1-\alpha \right) }.
$$

\item Considering $\Psi \left( x\right) =x,$ $a=c$ and taking the limit $ \beta \rightarrow 0$ in both sides of Eq.(\ref{eq22}), we obtain a Leibniz type rule 
for the Chen fractional derivative
\begin{eqnarray}
^{\mathbf{H}}\mathfrak{D}_{c+}^{\alpha ,0;x}\left( fg\right) \left( x\right) &=&\overset{\infty }{\underset{m=0}{\sum }}\left( 
\begin{array}{c}
\alpha  \\ 
m%
\end{array}%
\right) f^{\left( m\right) }\left( x\right) \text{ }^{\mathbf{EN}}D_{c+}^{\alpha
-m}g\left( x\right)   \notag \\
&&+\underset{k=0}{\overset{\infty }{\sum }}\left( 
\begin{array}{c}
\alpha -1 \\ 
k%
\end{array}%
\right) \mathcal{I}_{c+}^{\left( 1-\alpha \right) +k}g\left( c\right) \left( f^{\left( k\right) }\left( x\right) -f^{\left( k\right) }\left( c\right) \right) 
\frac{\left( x-c\right) ^{-\alpha }}{\Gamma \left( 1-\alpha \right) }  \notag \\
&=&\overset{\infty }{\underset{m=0}{\sum }}\left( 
\begin{array}{c}
\alpha  \\ 
m%
\end{array}%
\right) f^{\left( m\right) }\left( x\right) \text{ }^{\mathbf{EN}}D_{c+}^{\alpha -m}g\left( x\right) =\text{ }^{\mathbf{EN}}D_{c+}^{\alpha }\left( fg\right) \left(x\right) .
\end{eqnarray}

\end{enumerate}


\section{Concluding remarks}

In this paper, we investigate Lemma \ref{lemma1} and Lemma \ref{lemma1} to calculate the derivative of the product of two functions. In this sense, we present the 
Leibniz rule for the $\Psi-$H fractional derivative, since it had not yet been investigated. From the choice of the function $\Psi(\cdot)$ and the limits $\beta \rightarrow 1$ 
and $\beta \rightarrow 0$, we have obtained a class of Leibniz and Leibniz type rules of their respective fractional derivatives.

From the Leibniz type rule presented here for the first time, the next step of the research is to investigate the Ulam-Hyers stabilities of solutions of fractional differential 
equations of the functional, impulsive and evolution types, among others. We believe that from the formulation of this Leibniz type rule, we anticipate the possibility of 
investigating other problems related to the fractional calculus.

\section{Acknowledgment}

I support financial of the PNPD-CAPES scholarship of the Pos-Graduate Program in Mathematics Applied IMECC-Unicamp.

\bibliography{ref}
\bibliographystyle{plain}

\end{document}